\documentclass[12pt, reqno]{amsart}
\usepackage{amsmath, amsthm, amscd, amsfonts, amssymb, graphicx}
\usepackage[bookmarksnumbered,plainpages]{hyperref}

\newtheorem{theorem}{Theorem}[section]

\newtheorem{proposition}[theorem]{Proposition}
\newtheorem{corollary}[theorem]{Corollary}
\theoremstyle{definition}

\theoremstyle{remark}
\newtheorem{remark}[theorem]{Remark}
\numberwithin{equation}{section}
\newcommand{\ip}[2]{\langle#1,#2\rangle}

\begin{document}
\title[Reverse Cauchy--Schwarz inequalities]{Reverse Cauchy--Schwarz inequalities for positive
$C^*$-valued sesquilinear forms}
\author[M.S. Moslehian, L.E. Persson]{Mohammad Sal Moslehian$^1$ and Lars-Erik Persson$^2$}
\address{Department of Pure Mathematics, Ferdowsi University of Mashhad, P.O. Box 1159,
Mashhad 91775, Iran;\newline Centre of Excellence in Analysis on
Algebraic Structures (CEAAS), Ferdowsi University of Mashhad, Iran.}
\email{moslehian@ferdowsi.um.ac.ir and
moslehian@ams.org}

\address{Department of Mathematics, Lule\aa\ University of Technology, SE 97187, Lule\aa, Sweden.}
\email{larserik@sm.luth.se}


\begin{abstract}
We prove two new reverse Cauchy--Schwarz inequalities of additive
and multiplicative types in a space equipped with a positive
sesquilinear form with values in a $C^*$-algebra. We apply our
results to get some norm and integral inequalities. As a
consequence, we improve a celebrated reverse Cauchy--Schwarz
inequality due to G.~P\'olya and G.~Szeg\"o.
\end{abstract}

\maketitle


\section{Introduction}

The probably first reverse Cauchy--Schwarz inequality for positive
real numbers $a_1, \cdots, a_n$ is the following one (see \cite[p. 57
and 213-214]{P-S} and \cite[p. 71-72 and 253-255]{B}):

\noindent {\bf Theorem [G.~P\'olya and G.~Szeg\"o (1925)].} Let
$a_1, \cdots, a_n$ and $b_1, \cdots, b_n$ be positive real numbers.
If
$$0 < a \leq a_i \leq A < \infty\,, 0 < b \leq b_i \leq B <
\infty$$ for some constants $a, b, A, B$ and all $1 \leq i \leq n$,
then
\begin{eqnarray}\label{ps}
\sum_{i=1}^n a_i^2\,\sum_{i=1}^n b_i^2 \leq \frac{(ab+AB)^2}{4abAB}
\left(\sum_{i=1}^n a_ib_i\right)^2\,.
\end{eqnarray}
The inequality is sharp in the sense that $1/4$ is the best possible
constant.

\noindent We remark that \eqref{ps} can be obviously rewritten in
the following equivalent form
\begin{eqnarray}\label{o}
\sum_{i=1}^n a_i^2\,\sum_{i=1}^n b_i^2- \left(\sum_{i=1}^n
a_ib_i\right)^2\leq \frac{(AB-ab)^2}{4abAB}\left(\sum_{i=1}^n
a_ib_i\right)^2\,.
\end{eqnarray}
We say that \eqref{ps} is the multiplicative form of the
P\'olya--Szeg\"o inequality and that \eqref{o} is the additive form.

There exist a lot of generalizations of this classical
inequality. For example, Chapter 5 in \cite{A} (36 pages) is devoted
only to such reversed discrete Cauchy--Schwarz inequalities. Also
similar results for integrals, isotone functionals as well as
generalizations in the setting of inner product spaces are today
well-studied and understood; see the books \cite{DRA1} and
\cite{DRA2}. Moreover, C.P.~Niculescu \cite{NIC} and M.~Joita
\cite{JOI} have proved some reverse Cauchy--Schwarz inequalities in
the framework of $C^*$-algebras. We also refer to another
interesting even newer paper by D.~Ilisevi\'c and S.~Varosanec
\cite{I-V} of this type; see also the book of T. Furuta, J.M. Hot,
J.E. Pe\v cari\'c and Y. Seo \cite{F-H-P-S} and references therein.

In this paper we continue and complement this research by proving
some new generalizations of both \eqref{ps} and \eqref{o} in a
similar framework (see Theorem \ref{main1} and Theorem \ref{main2}).
We also apply our results to get some norm and integral
inequalities. As a consequence, we improve inequality \eqref{o}.

\section{Preliminaries}

A \emph{$C^*$-algebra} is a Banach $*$-algebra $({\mathfrak A}, \|
\cdot \|)$ such that $\| a^*a \| = \| a \|^2$ for each $a \in
{\mathfrak A}$. Recall that $a \in {\mathfrak A}$ is called
\emph{positive} (we write $a \geq 0$) if $a=b^*b$ for some $b \in
{\mathfrak A}$. If $a \in {\mathfrak A}$ is positive, then there is
a unique positive $b \in {\mathfrak A}$ such that $a=b^2$; such an
element $b$ is called the positive square root of $a$ and denoted by
$a^{1/2}$. For every $a \in {\mathfrak A}$, the positive square root
of $a^*a$ is denoted by $|a|$. For two self-adjoint elements $a, b$
one can define a partial order $\leq$ by $$a\leq b \Leftrightarrow
b-a \geq 0.$$ For $a \in A$, by $\mbox{Re}~a$ we denote $\frac{a +
a^*}{2}$.

Let ${\mathfrak A}$ be a $C^*$-algebra and let ${\mathfrak
X}$ be an algebraic right ${\mathfrak A}$-module which is a complex
linear space with $(\lambda x)a=x(\lambda a)=\lambda (xa)$ for all
$x\in {\mathfrak X}$, $a\in {\mathfrak A}$, $\lambda\in {\mathbb
C}$. The space ${\mathfrak X}$ is called a \emph{(right) semi-inner
product ${\mathfrak A}$-module} (or \emph{semi-inner product
$C^*$-module over the $C^*$-algebra ${\mathfrak A}$}) if there exists an
${\mathfrak A}$-valued inner product, i.e., a mapping
$\ip{\cdot}{\cdot} \colon {\mathfrak X}\times {\mathfrak X}\to
{\mathfrak A}$ satisfying
\begin{enumerate}
\item[{\rm (i)}] $\ip{x}{x}\geq 0$,

\item[{\rm (ii)}] $\ip{x}{\lambda y+z} = \lambda\ip{x}{y}+
\ip{x}{z}$,

\item[{\rm (iii)}] $\ip{x}{ya} =\ip{x}{y}a$,

\item[{\rm (iv)}] $\ip{y}{x}=\ip{x}{y}^*$,
\end{enumerate}
for all $x, y, z \in {\mathfrak X}$, $a\in {\mathfrak A}$,
$\lambda\in {\mathbb C}$. Moreover, if
\begin{enumerate}
\item[{\rm (v)}] $x=0$ whenever $\ip{x}{x} =0$,
\end{enumerate}
then ${\mathfrak X}$ is called an \emph{inner product ${\mathfrak
A}$-module} (\emph{inner product $C^*$-module over the $C^*$-algebra
${\mathfrak A}$}). In this case $\| x \| := \sqrt{\| \langle x, x
\rangle \|}$ gives a norm on ${\mathfrak X}$ making it into a normed
space. If this normed space is complete, then ${\mathfrak X}$ is
called a \emph{Hilbert ${\mathfrak A}$-module} (\emph{Hilbert
$C^*$-module over the $C^*$-algebra ${\mathfrak A}$}). A left inner
product ${\mathfrak A}$-module can be defined analogously. Any inner
product (resp.~Hilbert) space is a left inner product
(resp.~Hilbert) $\mathbb{C}$-module and any $C^*$-algebra
${\mathfrak A}$ is a right Hilbert $C^*$-module over itself via
$\langle a, b \rangle =a^*b$, for all $a,b \in {\mathfrak A}$. For
more details on inner product $C^*$-modules see \cite{LAN} and
\cite{M-T}.

The Cauchy--Schwarz inequality asserts that $\langle
x,y\rangle \langle y,x\rangle  \leq \|y\|^2\,\langle x,x\rangle $ in
a semi-inner product module ${\mathfrak X}$ over a $C^*$-algebra
${\mathfrak A}$; see \cite[Proposition 1.1]{LAN}. This is a
generalization of the classical Cauchy--Schwarz inequality stating
that if $x$ and $y$ are elements of a semi-inner product space, then
$|\langle x, y \rangle|^2 \leq \langle x, x\rangle · \langle y, y
\rangle$. There are mainly two types of reverse Cauchy--Schwarz
inequality. In the additive approach (initiated by N. Ozeki
\cite{OZE}) we look for an inequality of the form $k + |\langle x,
y\rangle|^2 \geq \langle x, x \rangle · \langle y, y \rangle$ for
some suitable positive constant $k$ (see also \eqref{o}). In the
multiplicative approach (initiated by G. Polya and G. Szeg\"o
\cite{P-S}) we seek for an appropriate positive constant $k$ such
that $|\langle x, y \rangle|^2 \geq k\langle x, x \rangle · \langle
y, y \rangle$ (see also \eqref{ps}).

In the next section we prove and discuss some reverse
Cauchy--Schwarz inequalities of additive and multiplicative types in
a linear space equipped with a positive sesquilinear form with
values in a $C^*$-algebra. For a comprehensive account on
Cauchy--Schwarz inequality and its various inverses we refer the
reader to books \cite{DRA1} and \cite{DRA2}.

\section{The main results}

Let ${\mathfrak A}$ be a $C^*$-algebra and let ${\mathfrak X}$ be a
linear space. By an ${\mathfrak A}$-valued positive sesquilinear
form we mean a mapping $\langle \cdot, \cdot \rangle : {\mathfrak X}
\times {\mathfrak X} \to {\mathfrak A}$ which is linear in the first
variable and conjugate linear in the second and fulfills $\langle x,
x \rangle \geq 0 \quad (x \in {\mathfrak X})$.

\noindent Our first result of this section is the following additive
reverse Cauchy--Schwarz inequality:

\begin{theorem}\label{main1}
Let ${\mathfrak A}$ be a $C^*$-algebra and let ${\mathfrak X}$ be a
linear space equipped with an ${\mathfrak A}$-valued positive
sesquilinear form $\langle \cdot, \cdot \rangle : {\mathfrak X}
\times {\mathfrak X} \to {\mathfrak A}$. Suppose that $x, y \in
{\mathfrak X}$ are such that
\begin{eqnarray}\label{star1}
\langle x, y\rangle^*= \langle y, x\rangle
\end{eqnarray}
\begin{eqnarray}\label{com}
\langle y, y\rangle^{1/2}\langle x, y\rangle= \langle x,
y\rangle\langle y, y\rangle^{1/2}
\end{eqnarray}
\begin{eqnarray}\label{Re}
\mbox{Re}\, \langle \Omega y-x, x-\omega y\rangle \geq 0
\end{eqnarray}
for some $\omega, \Omega \in \mathbb{C}$. Then
\begin{eqnarray}\label{14}
\left |\langle x, x\rangle^{1/2}\right\langle y, y\rangle^{1/2} |^2
- \left |\langle y, x\rangle\right|^2 \leq \frac{1}{4} |\Omega -
\omega|^2 \langle y, y\rangle^2\,.
\end{eqnarray}
\end{theorem}
\begin{remark}
The constant $1/4$ in \eqref{14} can not in general be replaced by
some smaller numbers (see the proof for the special case considered
in Proposition \ref{prop1}).
\end{remark}

\begin{proof}
Set
\begin{eqnarray*}
D_1&:=& \mbox{Re}[(\Omega \langle y, y\rangle - \langle x, y\rangle)
(\langle y, x\rangle - \bar{\omega}\langle y, y\rangle)] \\
&=& - \mbox{Re}(\Omega \bar{\omega})\langle y, y\rangle^2 - \langle
x, y\rangle \langle y, x\rangle + \mbox{Re}[\Omega \langle y,
y\rangle \langle y, x\rangle + \bar{\omega} \langle x,
y\rangle\langle y, y\rangle]\,.
\end{eqnarray*}
It follows from (\ref{Re}) that
\begin{eqnarray*}
D_2:=\langle y, y\rangle^{1/2}\mbox{Re}\, \langle \Omega y-x,
x-\omega y\rangle \langle y, y\rangle^{1/2} \geq 0\,.
\end{eqnarray*}
Hence, by also using \eqref{com}, we find that
\begin{eqnarray*}
D_2&=& - \mbox{Re}(\Omega \bar{\omega})\langle y, y\rangle^2 -
\langle y,
y\rangle^{1/2}\langle x, x\rangle \langle y, y\rangle^{1/2}\\
&&+ \langle y, y\rangle^{1/2}\mbox{Re}[\Omega \langle y, x\rangle +
\bar{\omega} \langle x, y\rangle]\langle y, y\rangle^{1/2}\\
&=&  - \mbox{Re}(\Omega\bar{\omega})\langle y, y\rangle^2 - \langle
y, y\rangle^{1/2}\langle x, x\rangle \langle y,
y\rangle^{1/2}\\&&+\mbox{Re}[\Omega \langle y, y\rangle \langle y,
x\rangle + \bar{\omega} \langle x, y\rangle\langle y, y\rangle]\,.
\end{eqnarray*}
Therefore $D_1-D_2 \leq D_1$ and we conclude that
\begin{eqnarray*}
\langle y, y\rangle^{1/2}\langle x, x\rangle \langle y,
y\rangle^{1/2}- \langle x, y\rangle \langle y, x\rangle&\leq&
\mbox{Re}[(\Omega \langle y, y\rangle - \langle x, y\rangle)
(\langle y, x\rangle -
\bar{\omega}\langle y, y\rangle)]\\
&\leq& \frac{1}{4} \left|(\overline{\Omega}-\overline{\omega}) \langle y,
y\rangle\right|^2\,.
\end{eqnarray*}
The last inequality is obtained by applying (\ref{star1}) and the
elementary inequality $\mbox{Re}(u^*v) \leq \frac{1}{4}|u+v|^2
\quad(u, v \in {\mathfrak A})$ for $u=\overline{\Omega} \langle y,
y\rangle - \langle y, x\rangle$ and $v=\langle y, x\rangle -
\overline{\omega}\langle y, y\rangle$. The proof is complete.
\end{proof}

\noindent Our multiplicative reverse Cauchy--Schwarz inequality
reads as follows:

\begin{theorem}\label{main2}
Let ${\mathfrak A}$ be a $C^*$-algebra and let ${\mathfrak X}$ be a
linear space equipped with an ${\mathfrak A}$-valued positive
sesquilinear form $\langle \cdot, \cdot \rangle : {\mathfrak X}
\times {\mathfrak X} \to {\mathfrak A}$. Suppose that $x, y \in
{\mathfrak X}$ are such that {\rm(\ref{star1})} holds, $\langle x,
y\rangle$ is normal and {\rm(\ref{Re})} holds for some $\omega,
\Omega \in \mathbb{C}$ with $\mbox{Re}(\bar{\omega}\Omega)>0$. Then
\begin{eqnarray}\label{star}
\langle x, x\rangle^{1/2} \langle y, y\rangle^{1/2} + \langle y,
y\rangle^{1/2}\langle x, x\rangle^{1/2} \leq
\frac{|\Omega|+|\omega|}{[\mbox{Re}(\bar{\omega}\Omega)]^{1/2}}\,|\langle
x, y\rangle|\,.
\end{eqnarray}
\end{theorem}
\begin{proof}
It follows from (\ref{Re}) that
\begin{eqnarray*}
\mbox{Re} (\Omega \langle x, y\rangle^* + \bar{\omega}\langle x,
y\rangle) - \langle x, x\rangle -
[\mbox{Re}(\bar{\omega}\Omega)]\langle y, y\rangle \geq 0\,.
\end{eqnarray*}
Moreover, since $\langle x, y\rangle$ is normal, $Re(\langle x,
y\rangle) \leq |\langle x, y\rangle|$ so that
\begin{eqnarray}\label{1}
\frac{1}{[\mbox{Re}(\bar{\omega}\Omega)]^{1/2}} \langle x, x\rangle
+ [\mbox{Re}(\bar{\omega}\Omega)]^{1/2}\langle y, y\rangle &\leq&
\frac{\mbox{Re} (\Omega \langle x, y\rangle^* + \bar{\omega}\langle
x,
y\rangle)}{[\mbox{Re}(\bar{\omega}\Omega)]^{1/2}} \nonumber\\
&\leq& \frac{|\Omega| +
|\omega|}{[\mbox{Re}(\bar{\omega}\Omega)]^{1/2}}\, |\langle x,
y\rangle| \,.
\end{eqnarray}
Furthermore, the trivial estimate
$$\left(\frac{1}{[\mbox{Re}(\bar{\omega}\Omega)]^{1/4}} \langle x,
x\rangle^{1/2} - [\mbox{Re}(\bar{\omega}\Omega)]^{1/4}\langle y,
y\rangle^{1/2}\right)^2 \geq 0$$ implies that
\begin{eqnarray}\label{2}
\langle x, x\rangle^{1/2} \langle y, y\rangle^{1/2} &+& \langle y,
y\rangle^{1/2}\langle x, x\rangle^{1/2}  \nonumber\\
&\leq&\frac{1}{[\mbox{Re}(\bar{\omega}\Omega)]^{1/2}} \langle x,
x\rangle + [\mbox{Re}(\bar{\omega}\Omega)]^{1/2}\langle y, y\rangle.
\end{eqnarray}
By combining \eqref{1} and \eqref{2} we obtain \eqref{star} and the
proof is complete.
\end{proof}


\begin{proposition}\label{prop1}
Let $\varphi$ be a positive linear functional on a $C^*$-algebra
${\mathfrak A}$ and let $x, y \in {\mathfrak A}$ be such that
$$\mbox{Re}\,\varphi((x-\omega y)^*(\Omega y-x))\geq 0$$
for some $\omega, \Omega \in \mathbb{C}$.

(a) Then
\begin{eqnarray}\label{add1}
\varphi(x^*x)\varphi(y^*y) - |\varphi(y^*x)|^2 \leq
\frac{1}{4}|\Omega - \omega|^2 \varphi(y^*y)^2\, .
\end{eqnarray}

(b) Moreover, if for each $y \in {\mathfrak A}$ there exists an
element $z \in {\mathfrak A}$ such that $\varphi(z^*y)=0$, then the
constant $\frac{1}{4}$ is sharp in {\rm(\ref{add1})}.

(c) Furthermore, if $\mbox{Re}(\bar{\omega}\Omega)>0$, then
\begin{eqnarray}\label{mult1}
\varphi(x^*x)^{1/2}\varphi(y^*y)^{1/2}\leq \frac{1}{2}
\frac{|\Omega|+|\omega|}{[\mbox{Re}(\bar{\omega}\Omega)]^{1/2}}\,|\varphi(y^*x)|\,.
\end{eqnarray}
\end{proposition}
\begin{remark}
Proposition \ref{prop1} (a) and (b) is related to Theorem 1 of
\cite{C}, where the case with inner product spaces was considered.
\end{remark}

\begin{proof} (a) By defining $\langle u, v\rangle := \varphi(v^*u)$ one obtains a
positive sesquilinear form from ${\mathfrak A}\times {\mathfrak A}$
into $\mathbb{C}$. Equality (\ref{star1}) holds by \cite[p. 88]{MUR}
and equality (\ref{com}) is trivially fulfilled. Thus Theorem
\ref{main1} gives the additive type reverse Cauchy--Schwarz
inequality \eqref{add1} and Theorem \ref{main2} yields the
multiplicative type reverse Cauchy--Schwarz inequality
(\ref{mult1}). It remains to prove the sharpness assertion in (b). In fact, let $y \in {\mathfrak A}$ with $\varphi(y^*y)=1$.
Choose $z \in {\mathfrak A}$ with $\varphi(z^*z)=1$ and
$\varphi(z^*y)=0$. Put $x=\frac{\Omega+\omega}{2}y +
\frac{\Omega-\omega}{2} z$. Then
\begin{eqnarray*}
\varphi((x-\omega y)^*(\Omega
y-x))&=&\varphi\left(\left(\frac{\Omega-\omega}{2}y+\frac{\Omega-\omega}{2}z\right)^*
\left(\frac{\Omega-\omega}{2}y-\frac{\Omega-\omega}{2}z\right)\right)\\
&=&\left|\frac{\Omega-\omega}{2}\right|^2\varphi(y^*y-z^*z)\\
&=&0\,.
\end{eqnarray*}
Hence $\mbox{Re} \varphi((x-\omega y)^*(\Omega y-x))\geq 0$ holds.
If
\begin{eqnarray*}
\varphi(x^*x)\varphi(y^*y) - |\varphi(y^*x)|^2 \leq C |\Omega -
\omega|^2 \varphi(y^*y)^2\, ,
\end{eqnarray*}
for some $C\geq 0$, then
\begin{eqnarray*}
\left|\frac{\Omega-\omega}{2}\right|^2&=&
\varphi\left(\left|\frac{\Omega+\omega}{2}\right|^2y^*y +
\left|\frac{\Omega-\omega}{2}\right|^2z^*z\right)-\left|\frac{\Omega+\omega}{2}\right|^2\varphi(y^*y)\\
&=&\varphi(x^*x)-|\varphi(y^*x)|^2\\
&\leq& C |\Omega-\omega|^2\, ,
\end{eqnarray*}
from which we conclude that $1/4 \leq C$. The proof is complete.
\end{proof}


\begin{remark}\label{remark 3.6}
Let $\varphi$ be a positive linear functional on a $C^*$-algebra
${\mathfrak A}$, $x, y$ be self-adjoint elements of ${\mathfrak A}$
such that $\omega y \leq x \leq \Omega y$ for some scalars $\omega,
\Omega
>0$. Then
$$\varphi((x-\omega y)^*(\Omega y-x))\geq 0$$
\noindent if $xy=yx$ (in particular, when ${\mathfrak A}$ is
commutative), since
$$\varphi((x-\omega y)(\Omega y-x))=\varphi((\Omega y-x)^{1/2}(x-\omega y)(\Omega y-x)^{1/2})\geq 0\,.$$
In particular, if $x$ and $y$ are commuting strictly positive
elements of a unital $C^*$-algebra ${\mathfrak A}$, one may consider
$\omega=\frac{\inf \sigma(x)}{\sup \sigma(y)}$ and
$\Omega=\frac{\sup \sigma(x)}{\inf \sigma(y)}$, where $\sigma(a)$
denotes the spectrum of $a\in {\mathfrak A}$; cf. \cite{NIC}.
\end{remark}


\noindent We can apply Proposition \ref{prop1} and Remark
\ref{remark 3.6} to derive some new inequalities as well as some
well-known ones. We only give the following such results:

\begin{corollary}\label{cor1}
Let ${\mathcal H}$ be a Hilbert space and $T, S \in {\mathbb
B}({\mathcal H})$ be strictly positive operators such that $TS=ST$.
Then
\begin{eqnarray*}
\|Tx\|^2\|Sx\|^2 - |\langle Tx, Sx\rangle|^2 &\leq&
\Big(\frac{\sup\sigma(T)\sup\sigma(S)-\inf\sigma(T)\inf\sigma(S)}{2}\Big)^2\\
&&\min\Big\{\frac{\|Sx\|^4}{\sup\sigma(S)^2\inf\sigma(S)^2},
\frac{\|Tx\|^4}{\sup\sigma(T)^2\inf\sigma(T)^2}\Big\}
\end{eqnarray*}
and
\begin{eqnarray*}
\|Tx\| \|Sx\|\leq
\frac{1}{2}\left(\sqrt{\frac{\inf\sigma(T)\inf\sigma(S)}{\sup\sigma(T)\sup\sigma(S)}}+
\sqrt{\frac{\sup\sigma(T)\sup\sigma(S)}{\inf\sigma(T)\inf\sigma(S)}}\right)\,|\langle
Tx, Sx\rangle|
\end{eqnarray*}
for all $x\in {\mathcal H}$.
\end{corollary}
\begin{proof} It is sufficient to set ${\mathfrak A}={\mathbb B}({\mathcal H})$, to consider $\varphi(R):=\langle Rx,
x\rangle \quad (R \in {\mathbb B}({\mathcal H}))$ and to apply
Proposition \ref{prop1} and Remark \ref{remark 3.6}.
\end{proof}


\begin{corollary}\label{cor2}
Let $(X, \Sigma, \mu)$ be a probability space and $f, g \in
L^\infty(\mu)$ with $ 0 < a \leq f \leq A$, $0 < b \leq g \leq B$.
Then
\begin{eqnarray}\label{d1}
\int_X f^2 d\mu \int_X g^2 d\mu - \left(\int_X fg d\mu\right)^2
&\leq& \frac{(AB-ab)^2}{4}\min\Big\{\frac{1}{B^2b^2} \left(\int_X
g^2 d\mu\right)^2,\nonumber\\
&& \frac{1}{A^2a^2}\left(\int_X f^2 d\mu\right)^2\Big\}\,
\end{eqnarray}
and
\begin{eqnarray}\label{d2}
\left(\int_X f^2 d\mu\right)^{1/2} \left(\int_X g^2
d\mu\right)^{1/2}\leq \frac{1}{2}\left(\sqrt{\frac{ab}{AB}}+
\sqrt{\frac{AB}{ab}}\right)\,\int_X fg d\mu\,.
\end{eqnarray}
\end{corollary}
\begin{proof} It is enough to assume ${\mathfrak A}$ to be the commutative $C^*$-algebra $L^\infty(X,\mu)$, to consider
$\varphi(h):=\int_X h d\mu \quad (h \in L^\infty(\mu))$ , use
Proposition \ref{prop1} and Remark \ref{remark 3.6} and an obvious
symmetry argument.
\end{proof}

\begin{remark}
The second inequality of Corollary \ref{cor2} is due to
C.N.~Niculescu \cite{NIC}.
\end{remark}


By applying \eqref{d2} with a weighted counting measure
$\mu=\sum_{i=1}^nw_i\delta_i$, where $w_i$'s are positive numbers
and $\delta_i$'s are the Dirac delta functions, we obtain with
change of notation that if $a_1, \cdots, a_n$ and $b_1, \cdots, b_n$ 
satisfy the conditions in the P\'olya--Szeg\"o theorem, then
\begin{eqnarray*}
\sum_{i=1}^n a_i^2w_i\,\sum_{i=1}^n b_i^2w_i \leq
\frac{(AB+ab)^2}{4ABab} \left(\sum_{i=1}^n a_ib_iw_i\right)^2\,,
\end{eqnarray*}
which is the Greub--Rheinboldt inequality \cite{G-R}. In the same
way, from \eqref{d1} it follows that
\begin{align}\label{do}
\sum_{i=1}^n a_i^2w_i\,\sum_{i=1}^n b_i^2w_i&- \left(\sum_{i=1}^n
a_ib_iw_i\right)^2\leq
\\&\frac{(AB-ab)^2}{4}\min\Big\{\frac{1}{B^2b^2}
\left(\sum_{i=1}^n b_i^2w_i\right)^2\nonumber,
\frac{1}{A^2a^2}\left(\sum_{i=1}^n a_i^2w_i\right)^2\Big\}\,.
\end{align}
In particular, by using this inequality with $w_i=1\,\,(i=1, \cdots,
n)$, we get the following strict improvement of the P\'olya--Szeg\"o
inequality \eqref{o}.

\begin{corollary}\label{cor3}
Suppose that $a_1, \cdots, a_n$ and $b_1, \cdots, b_n$ are positive
real numbers such that $0 < a \leq a_i \leq A < \infty\,, 0 < b \leq
b_i \leq B < \infty$ for some constants $a, b, A, B$ and all $1 \leq
i \leq n$. Then
\begin{align}\label{good}
&\sum_{i=1}^n a_i^2\,\sum_{i=1}^n b_i^2- \left(\sum_{i=1}^n
a_ib_i\right)^2\leq
\nonumber\\&\frac{(AB-ab)^2}{4}\min\Big\{\frac{1}{A^2a^2}\left(\sum_{i=1}^n
a_i^2\right)^2, \frac{1}{B^2b^2} \left(\sum_{i=1}^n b_i^2\right)^2,
\frac{1}{abAB}\left(\sum_{i=1}^n a_ib_i\right)^2\Big\}\,.
\end{align}
Any of the constants in the bracket above can be the strictly least
one.
\end{corollary}
\begin{proof} The inequality \eqref{good} follows by just combining \eqref{o} with
\eqref{do}. Moreover, the proof of the final statement is as
follows: By choosing $n>1, a=1, a_i=A=n, b_i=b=B=1/n\,\,(1\leq i
\leq n)$ we find that the second constant is strictly less than the
first one and the third one. Analogously, by choosing $n>1,
a=a_i=A=1/n, b=1, b_i=B=n\,\, (1 \leq i \leq n)$ we find that the
first constant is strictly less than the second and the third one.
We also note that by the Schwarz inequality,
\begin{align}\label{sch}
\frac{1}{abAB}\left(\sum_{i=1}^n a_ib_i\right)^2\leq \frac{1}{abAB}
\sum_{i=1}^n a_i^2\sum_{i=1}^n b_i^2
\end{align}
and obviously
\begin{eqnarray*}
\frac{1}{abAB}\sum_{i=1}^n a_i^2\,\sum_{i=1}^n b_i^2\leq
\min\Big\{\frac{1}{A^2a^2}\left(\sum_{i=1}^n a_i^2\right)^2,
\frac{1}{B^2b^2} \left(\sum_{i=1}^n b_i^2\right)^2\Big\}
\end{eqnarray*}
if and only if
\begin{eqnarray}\label{=}
\frac{1}{Aa}\sum_{i=1}^n a_i^2=\frac{1}{Bb}\sum_{i=1}^n b_i^2\,.
\end{eqnarray}
We conclude that the third term is strictly less than the first two
whenever \eqref{=} holds and we have strict inequality in the Schwarz
inequality \eqref{sch}. Hence our claim is proved.
\end{proof}

\begin{remark}
By combining \eqref{do} with the Greub--Rheinboldt inequality we can
in a similar way derive also a weighted version of Corollary
\ref{cor3} (and, thus, also strictly improve the Greub--Rheinboldt
inequality).
\end{remark}

\textbf{Acknowledgement.} The authors would like to sincerely thank the referee for his/her useful comments.

\end{document}